\documentclass[11pt, leqno, letterpaper]{amsart}
 
\usepackage{graphicx}    
\usepackage{xcolor}

\usepackage[margin=1.2in]{geometry}

\usepackage{amsmath,amsthm,amsfonts,amssymb}

\newtheorem{theorem}{Theorem}

\newtheorem{proposition}[theorem]{Proposition}

\newtheorem{remark}[theorem]{Remark}

\begin{document}

\title{Asymptotic expansion for the Hartman-Watson distribution}
\author{Dan Pirjol}

\address
{School of Business\newline
\indent Stevens Institute of Technology\newline 
\indent Hoboken, NJ 07030}

\email{
dpirjol@gmail.com}
\date{20 February 2021}


\begin{abstract}
The Hartman-Watson distribution with density $f_r(t)=\frac{1}{I_0(r)}
\theta(r,t)$ with $r>0$ is a probability distribution defined on 
$t \geq 0$ which appears in several problems of applied probability. 
The density of this distribution is expressed in terms of an integral 
$\theta(r,t)$ which is difficult to evaluate numerically for small $t\to 0$.
Using saddle point methods, we obtain the first two terms 
of the $t\to 0$ expansion of $\theta(\rho/t,t)$ at fixed $\rho > 0$. 
An error bound is obtained by numerical estimates of the integrand, 
which is furthermore uniform in $\rho$.
As an application we obtain the leading asymptotics of the density of the
time average of the geometric Brownian motion as $t\to 0$.
This has the form $\mathbb{P}(\frac{1}{t}
\int_0^t e^{2(B_s+\mu s)} ds \in da ) \sim (2\pi t)^{-\frac12} g(a,\mu)
e^{-\frac{1}{t} J(a)}\frac{da}{a}$, with an exponent $J(a)$ which reproduces 
the known result obtained previously using Large Deviations theory.
\end{abstract}

\maketitle

\section{Introduction}
\label{sec:1}

The Hartman-Watson distribution was introduced in the context of
directional statistics \cite{Hartman1974} and was studied further
in relation to the first hitting time of certain diffusion processes
\cite{Kent1982}. 
This distribution has received considerable
attention due to its relation to the law of the time integral of the
geometric Brownian motion, see Eq.~(\ref{AB}) below \cite{Yor1992}. 

The Hartman-Watson distribution is given in terms of the function 
$\theta(r,t)$ defined as
\begin{equation}\label{thetadef}
\theta(r,t) = \frac{r}{\sqrt{2\pi^3 t}} e^{\frac{\pi^2}{2t}}
\int_0^\infty e^{-\frac{\xi^2}{2t}} e^{-r \cosh\xi} \sinh\xi \sin \frac{\pi\xi}{t}
d\xi\,.
\end{equation}
The normalized function $f_r(t) = \frac{\theta(r,t)}{I_0(r)}$ defines the
density of a random variable taking values on the positive real axis
$t \geq 0$, called the Hartman-Watson law \cite{Hartman1974}.

The function $\theta(r,t)$ is related to the distribution of the time-integral of a geometric Brownian motion.
Define
\begin{equation}\label{Adef}
A_t^{(\mu)} = \int_0^t e^{2(B_s +\mu s)} ds 
\end{equation}
with $B_t$ a standard Brownian motion and $\mu\in \mathbb{R}$.
The time-integral of the geometric Brownian motion is relevant for the pricing of Asian options in the 
Black-Scholes model \cite{DufresneReview}, and appears also in actuarial science \cite{Boyle2008}
and in the study of diffusion processes 
in random media \cite{Comtet1998}.
Other applications to mathematical finance are related to the distributional properties of stochastic
volatility models with log-normally distributed volatility, such as the
$\beta=1$ log-normal SABR model \cite{SABRbook,Cai2017} and the Hull-White model 
\cite{Gul2006,Gul2010}. 

An explicit result for the joint distribution of $(A_t^{(\mu)}, B_t)$
was given by Yor \cite{Yor1992}
\begin{eqnarray}\label{AB}
\mathbb{P}(A_t^{(\mu)} \in du, B_t +\mu t \in dx) = e^{\mu x - \frac12 \mu^2 t}
\exp\left(  - \frac{1+e^{2x}}{2u} \right) \theta(e^x/u,t) 
\frac{du dx}{u}\,, 
\end{eqnarray}
where the function $\theta(r,t)$ is given by (\ref{thetadef}).

The Yor formula (\ref{AB}) yields also the density of $A_t^{(\mu)}$ by integration over 
$B_t$. The usefulness of this approach is limited by 
issues with the numerical evaluation of the integral in (\ref{thetadef}) for
small $t$, due to the rapidly oscillating factor $\sin(\pi\xi/t)$ 
\cite{Barrieu2004,Boyle2008}. Alternative numerical approaches which
avoid this issue were presented in \cite{Bernhart2015,Ishyiama2005}.

For this reason considerable effort has been 
devoted to applying analytical methods to simplify Yor's formula. 
For particular cases $\mu=0,1$ simpler expressions as single integrals
have been obtained for the density of $A_t^{(\mu)}$ by Comtet et al 
\cite{Comtet1998} and Dufresne \cite{Dufresne2001}. 
See the review article by Matsumoto and Yor \cite{Matsumoto2005}
for an overview of the literature. 

In the absence of simple exact analytical results, it is important to have
analytical expansions of $\theta(r,t)$ in the small-$t$ region. 
Such an expansion has been derived by Gerhold in \cite{Gerhold2011},
by a saddle point analysis of the Laplace inversion integral of the
density $f_r(t)$.

In this paper we derive the $t\to 0$ asymptotics of the function
$\theta(r,t)$ at fixed $rt =\rho$ using the saddle point method. 
This regime is important for the study of the small-$t$ asymptotics of
the density of $\frac{1}{t} A_t^{(\mu)}$ following from the Yor formula
(\ref{AB}). The resulting expansion turns out to give also a good
numerical approximation of $\theta(r,t)$ for all $t\geq 0$.
The expansion has the form
\begin{equation}\label{thetaexp}
\theta(\rho/t,t) = \frac{1}{2\pi t} e^{-\frac{1}{t}
\left( F(\rho) - \frac{\pi^2}{2}\right) } 
\left( G(\rho)  + G_1(\rho) t + O(t^2) \right)\,.
\end{equation}
The leading order term proportional to $G(\rho)$ is given in 
Proposition \ref{prop:1}, and the subleading correction 
$G_1(\rho)$ is given in the Appendix.

The paper is structured as follows. In Section~\ref{sec:2} we present the
leading order asymptotic expansion of the integral $\theta(r,t)$ at fixed 
$\rho = rt$. The main result is Proposition \ref{prop:1}.
An explicit error bound for the leading term in the asymptotic expansion
of $\theta(\rho/t,t)$ is given, guided by numerical tests, which is furthermore uniformly bounded in $\rho$.
In Section~\ref{sec:3} we compare this asymptotic result with existing
results in the literature on the small $t$ expansion of the Hartman-Watson 
distribution \cite{Barrieu2004,Gerhold2011}.
In Section~\ref{sec:4} we apply the expansion to obtain the leading $t\to 0$ asymptotics of the density of the time averaged geometric Brownian motion 
$\mathbb{P}(\frac{1}{t} \int_0^t e^{2(B_s+\mu s)} ds \in da )=f(a,t) \frac{da}{a}$. 
The leading asymptotics of this density has the form 
$f(a,t) \sim (2\pi t)^{-\frac12} g(a,\mu) e^{-\frac{1}{t} J(a)}$. 
The exponential factor $J(a)$ reproduces a known
result obtained previously using Large Deviations theory \cite{MLP,SIFIN}.
An Appendix gives an explicit result for the subleading correction.

\section{Asymptotic expansion of $\theta(\rho/t, t)$ as $t\to 0$}
\label{sec:2}

We study here the asymptotics of $\theta(r,t)$ as $t\to 0$ at fixed 
$rt = \rho$. This regime is different from that considered in 
\cite{Gerhold2011}, which studied the asymptotics of $\theta(r,t)$ as $t\to 0$ 
at fixed $r$.  

\begin{proposition}\label{prop:1}
The $t\to 0$asymptotics of the Hartman-Watson integral $\theta(\rho/t,t)$ is
\begin{equation}
\theta(\rho/t,t) = \frac{1}{2\pi t}  e^{-\frac{1}{t}
\left( F(\rho) - \frac{\pi^2}{2}\right) } 
\left(G(\rho)  + G_1(\rho) t+ O(t^2) \right)\,,\quad t\to 0\,.
\end{equation}

The function $F(\rho)$ is given by 
\begin{eqnarray}\label{Fsol}
F(\rho) = \left\{
\begin{array}{cc}
 \frac12 x_1^2- \rho \cosh x_1 + \frac{\pi^2}{2} & \,, 0 < \rho < 1 \\
-\frac12 y_1^2+ \rho \cos y_1 + \pi y_1 & \,, \rho > 1 \\
\frac{\pi^2}{2} - 1 & \,, \rho = 1  \\
\end{array}
\right.
\end{eqnarray}
and the function $G(\rho)$ is given by
\begin{eqnarray}\label{Gsol}
G(\rho) = \left\{
\begin{array}{cc}
\frac{\rho \sinh x_1}{\sqrt{\rho \cosh x_1 - 1}} & \,,
   0 <\rho<1\\
\frac{\rho sin y_1}{\sqrt{1+\rho \cos y_1 }} & \,, \rho>1\\
 \sqrt3  & \,, \rho = 1  \\
\end{array}
\right.
\end{eqnarray}

Here $x_1$ is the solution of the equation
\begin{equation}\label{eqx1}
\rho \frac{\sinh x_1}{x_1} = 1
\end{equation}
and $y_1$ is the solution of the equation 
\begin{equation}\label{eqy1}
y_1 + \rho\sin y_1 = \pi\,.
\end{equation}

The subleading correction is $G_1(\rho)= \frac{1}{2} G(\rho) \tilde g_2(\rho)$ 
where $\tilde g_2(\rho)$ is given in explicit form in the Appendix.

\end{proposition}

Plots of the functions $F(\rho)$ and $G(\rho)$ are shown in 
Figure~\ref{Fig:FG}.
The properties of these functions are studied in more detail 
in Section \ref{sec:FG}.

\begin{figure}
    \centering
   \includegraphics[width=2.4in]{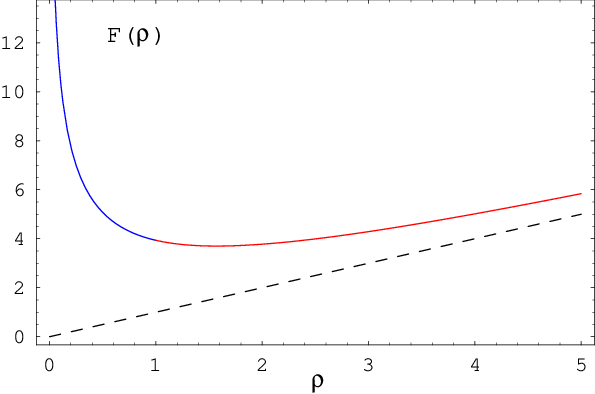}
   \includegraphics[width=2.4in]{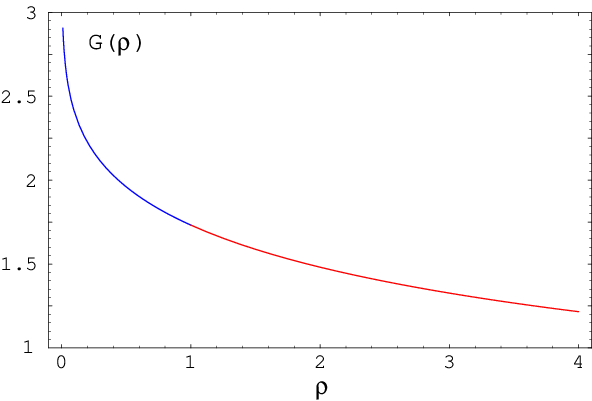}
    \caption{Left: Plot of $F(\rho)$ given in Eq.~(\ref{Fsol}). The two branches
in Eq.~(\ref{Fsol}) are shown as the blue and red curves, respectively. 
The function $F(\rho)$ has a minimum at $\rho=\frac{\pi}{2}$, with
$F(\frac{\pi}{2})=\frac{3\pi^2}{8}$. The
dashed line shows the asymptotic line $F(\rho) \sim \rho$ for 
$\rho\to \infty$.
Right: Plot of $G(\rho)$ defined in Eq.~(\ref{Gsol}). 
}
\label{Fig:FG}
 \end{figure}

\begin{proof}[Proof of Proposition \ref{prop:1}]

The function $\theta(\rho/t,t)$ can be written as 
\begin{eqnarray}
&& \theta(\rho/t, t) = 
\frac{\rho/t}{\sqrt{2\pi^3 t}} e^{\frac{\pi^2}{2t}}
\int_0^\infty e^{-\frac{1}{t} [\frac12 \xi^2 + \rho \cosh \xi]} \sinh \xi
\sin \frac{\pi\xi}{t} d\xi \\
&& = \frac{\rho/t}{\sqrt{2\pi^3 t}} e^{\frac{\pi^2}{2t}} \frac12
\int_{-\infty}^\infty e^{-\frac{1}{t} [\frac12 \xi^2 + \rho \cosh \xi]} \sinh \xi
\sin \frac{\pi\xi}{t} d\xi \nonumber \\
&& = \frac{\rho}{\sqrt{2\pi^3 t^3}} e^{\frac{\pi^2}{2t}} \frac{1}{4i}
(I_+ (\rho,t) - I_-(\rho,t))\,, \nonumber
\end{eqnarray}
with
\begin{equation}
I_\pm(\rho,t) := \int_{-\infty}^\infty 
e^{-\frac{1}{t} [\frac12 \xi^2 + \rho \cosh \xi
\mp i\pi \xi]} \sinh \xi d\xi \,.
\end{equation}

These integrals have the form
$\int_\alpha^\beta e^{-\frac{1}{t} h(\xi)} g(\xi) d\xi$
with 
$h_\pm(\xi) = \frac12 \xi^2 + \rho \cosh \xi \mp i\pi \xi$.

The asymptotics of $I_\pm(t)$ as $t\to 0$ can be obtained using the saddle-point
method, see for example 
Sec.~4.6 in Erd\'elyi \cite{Erdelyi} and Sec.~4.7 of Olver \cite{Olver}.

We present in detail the asymptotic expansion for $t\to 0$ of the integral
\begin{equation}
I_+(\rho,t) = \int_{-\infty}^\infty e^{-\frac{1}{t} h(\xi)} \sinh \xi d\xi
\end{equation}
where we denote for simplicity $h(\xi) = h_+(\xi)=\frac12 \xi^2+\rho \cosh \xi
-i\pi \xi$.
The integral $I_-(\rho,t)$ is treated analogously.

\begin{figure}
    \centering
   \includegraphics[width=1.5in]{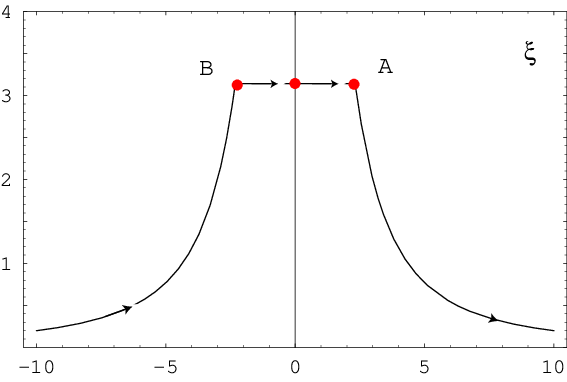}
   \includegraphics[width=1.5in]{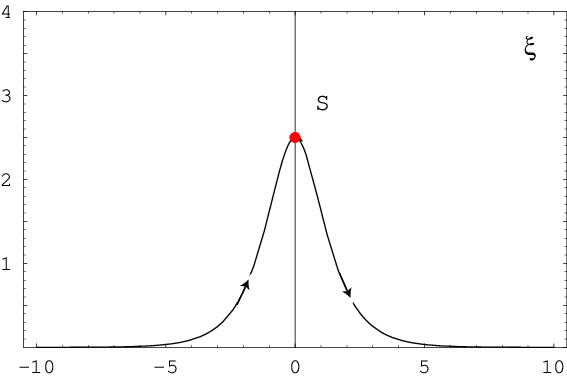}
   \includegraphics[width=1.5in]{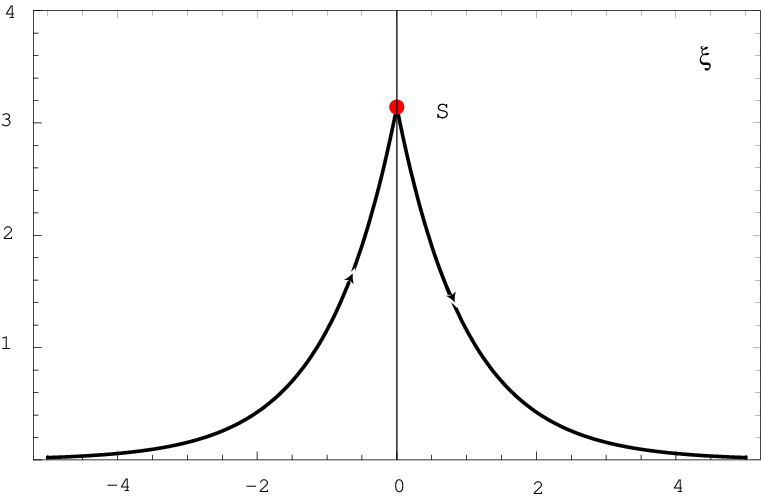}
    \caption{ Integration contours for $I_+(\rho,t)$ in the $\xi$ complex plane
for the application of the asymptotic expansion. The contours for 
$I_-(\rho,t)$ are obtained by changing the sign of $y$ (reflection in the
real axis). The red dots show the saddle points. 
Left: contour for $0<\rho<1$. The contour passes through the saddle points 
$B(\xi=-x_1+i\pi)$ and $A(\xi=x_1+i\pi)$.
Middle: contour for $\rho>1$. 
The contour passes through the saddle point $S(\xi=i y_1)$.
Right: the contour for $\rho=1$ passes through the fourth order saddle point at $S(\xi = i\pi)$.}
\label{Fig:contours}
 \end{figure}

i) $0<\rho <1$. The saddle points are given by the solution of the
equation $h'(\xi)=0$.
There are three saddle points at $\xi = \{i\pi, i\pi \pm x_1\}$ with $x_1$
the solution of the equation $x_1 - \rho \sinh x_1=0$. 
The second derivative of $h$
at these points is $h''(i\pi) = 1 - \rho > 0$, and $h''(i\pi \pm x_1)=1-\rho
\cosh x_1 < 0$. 

The contour of integration is deformed from the real axis
to the contour shown in the left panel of Fig.~\ref{Fig:contours}, consisting
of three arcs of curves of steepest descent passing 
through the three saddle points. Along these arcs we have 
$Im[h(\xi)]= x (y - \pi) + \rho \sinh x \sin y = 0$. 
Denoting $\xi=x + iy$, the path is given by 
\begin{eqnarray}
y= \left\{
\begin{array}{cc}
\pi & \,, |x| \leq x_1 \\
y_0(x) & \,, |x| > x_1 \\
\end{array}
\right.
\end{eqnarray}
where $y_0(x)$ is the positive solution of the equation 
$\rho \sinh x \sin y = (\pi-y) x$.

The integral is a sum of three integrals along each piece of the path. 
The real part of the integrand is odd and the imaginary
part is even under $x \to -x$. This follows from noting that we have 
$\mbox{Re} h(-x+iy) =
\mbox{Re} h(x+iy)$ and $\sinh (x+iy) = \sinh x \cos y + i \sin y \cosh x$.
This implies that i) the integral from B to A vanishes because the integrand 
is odd, and ii) the real parts of the integrals along $(-\infty, B]$ and 
$[A,\infty)$ are equal and of opposite sign, and their imaginary parts are equal.
This gives
\begin{equation}\label{thetaIA}
\theta(\rho/t,t) = \frac{\rho}{\sqrt{2\pi^3 t^3}}e^{\frac{\pi^2}{2t}} 
\frac{1}{2i} \left( I_{(-\infty,B]} + I_{[A,\infty)} \right)
= \frac{\rho}{\sqrt{2\pi^3 t^3}}e^{\frac{\pi^2}{2t}} 
\mbox{Im} I_{[A,\infty)}\,.
\end{equation}

Thus it is sufficient to evaluate only $\mbox{Im} I_{[A,\infty)}$.
This integral is written as
\begin{equation}\label{2.10}
I_{[A,\infty)}(\rho,t) = \int_{A}^\infty e^{-\frac{1}{t} h(\xi)} \sinh \xi d\xi =
e^{-\frac{1}{t} h(A)} \int_0^\infty 
e^{-\frac{1}{t} \tau} \frac{\sinh \xi}{h'(\xi)} d\tau \,,
\end{equation}
where we defined $\tau = h(\xi) - h(A)$. This is expanded around $\xi=A$ as
\begin{equation}\label{tauexp}
\tau = h(\xi) - h(A) = a_2 (\xi-A)^2 + a_3 (\xi-A)^3 + a_4 (\xi-A)^4 + O((\xi-A)^5)
\end{equation}
with $a_2=\frac12 h''(A), a_3 = \frac16 h'''(A), \cdots$.
Noting $h''(A) = 1 - \rho \cosh x_1 <0$, this series is inverted as
\begin{equation}\label{xiexp}
\xi - A = -i \sqrt{\frac{2\tau}{|h''(A)|}} + O(\tau)
\end{equation}

The second factor in the integrand of (\ref{2.10}) is also expanded in $\xi-A$ as
\begin{equation}
\frac{\sinh \xi}{h'(\xi)} = \frac{\sinh A + \cosh A (\xi-A) + O((\xi-A)^2) }
{h''(A)(\xi-A) + O((\xi-A)^2) }
= \frac{\sinh A}{h''(A)} \frac{1}{\xi-A} \left(1 + O(\xi-A) \right)
\end{equation}
Substituting here the expansion (\ref{xiexp}) gives 
\begin{equation}
\frac{\sinh \xi}{h'(\xi)} =  i\frac{\sinh x_1}{|h''(A)|} 
\sqrt{\frac{|h''(A)|}{2}} \frac{1}{\sqrt{\tau}} + O(\tau^0)
\end{equation}

More generally 
\begin{equation}\label{Imgexp1}
g(\tau) = \frac{\sinh \xi}{h'(\xi)} = \frac{g_0}{\sqrt{\tau}} + g_1 + g_2 \sqrt{\tau} +
O(\tau)\,.
\end{equation}
The inequality $h''(A)<0$ implies that $g_0, g_2 , \cdots $ are imaginary,
while $g_1, g_3, \cdots $ are real. 

By Watson's lemma \cite{Olver}, the resulting expansion can be integrated term 
by term. The leading asymptotic contribution to $I_{[A,\infty)}(\rho,t)$ is 
\begin{eqnarray}
&& \mbox{Im} I_{[A,\infty)}(\rho,t) = e^{-\frac{1}{t}h(A)} 
\mbox{Im} \int_0^\infty e^{-\frac{1}{t} \tau} \frac{\sinh \xi}{h'(\xi)} d\tau \\
&& \qquad =
\frac{\sinh x_1}{\sqrt{2|h''(A)|}} e^{-\frac{1}{t}h(A)} 
\left( \sqrt{\pi t} + O(t^{3/2}) \right)\,. \nonumber
\end{eqnarray}
The leading order term integral was evaluated as
$\int_0^\infty e^{-\frac{1}{t}\tau}
\frac{d\tau}{\sqrt{\tau}}=\sqrt{\pi t}$.
Substituting into (\ref{thetaIA}) reproduces the quoted result by identifying 
$h(A)=F(\rho)$. 
Since $g_1, g_3, \cdots$ are real, the $O(\tau^0)$ term in (\ref{Imgexp1}) does not 
contribute to $Im I_{[A,\infty)}$, and the leading correction comes from
the $O(\tau^{1/2})$ term.
This is given in explicit form in the Appendix.

ii) $\rho > 1$. There are several saddle points along the imaginary axis. 
We are
interested in the saddle point at $\xi=i y_1$ with $0 < y_1\leq \pi$ 
the solution of (\ref{eqy1}). At this point the second derivative of $h$ is 
$h''(i y_1) = 1 + \rho \cos y_1 >0$. 

Deform the $\xi:(-\infty,+\infty)$ integration contour into the curve in
the middle panel of Fig. ~\ref{Fig:contours}. This is a steepest descent
curve $\mbox{Im}(h(\xi))=0$, given by $y_0(x)$, the positive solution of the equation 
$\rho \sinh x \sin y = (\pi-y) x$.
The contour passes through the saddle 
point $S$ at $\xi = i y_1$.

The integral $I_+(\rho,t) = \int_{-\infty}^S +\int_S^{+\infty}$
is the sum of the two integrals on the two halves of the contour.
As in the previous case, the sum is imaginary since $h(\xi)$ is real along 
the contour, and $Re[h(\xi)]$ is even in $x$. Noting that 
$\sinh(x+iy)=\sinh x \cos y + i \sin y \cosh x$, the first term 
is odd in $x$ and cancels out, and only the second (imaginary) term
gives a contribution. We have a result similar to (\ref{thetaIA})
\begin{equation}\label{thetaIS}
\theta(\rho/t,t) = \frac{\rho}{\sqrt{2\pi^3 t^3}}e^{\frac{\pi^2}{2t}} 
\frac{1}{2i} \left( I_{(-\infty,S]} + I_{[S,\infty)} \right)
= \frac{\rho}{\sqrt{2\pi^3 t^3}}e^{\frac{\pi^2}{2t}} 
\mbox{Im} I_{[S,\infty)}\,.
\end{equation}

As before, it is sufficient to evaluate only the $[S,\infty)$ integral, 
which is
\begin{equation}\label{IScase2}
I_{[S,\infty)}(\rho,t) = \int_{S}^\infty e^{-\frac{1}{t} h(\xi)} \sinh \xi d\xi =
e^{-\frac{1}{t} h(S)} \int_0^\infty 
e^{-\frac{1}{t} \tau} \frac{\sinh \xi}{h'(\xi)} d\tau
\end{equation}
where we introduced $\tau = h(\xi) - h(S) \geq 0$. 
This difference is expanded around $\xi=S$ as
\begin{equation}
\tau = h(\xi) - h(S) = \frac12 h''(S) (\xi-S)^2 + O((\xi-S)^3)
\end{equation}
which is inverted as (recall $h''(S)>0$)
\begin{equation}\label{ximexpS}
\xi - S = \sqrt{\frac{2\tau}{h''(S)}} + O(\tau)
\end{equation}
The integrand is also expanded in $\xi-S$ as
\begin{equation}
\frac{\sinh \xi}{h'(\xi)} = \frac{\sinh S + \cosh S (\xi-S) + O((\xi-S)^2) }
{h''(S)(\xi-S) + O((\xi-S)^2) }
= \frac{\sinh S}{h''(S)} \frac{1}{\xi-S} \left(1 + O((\xi-S)) \right)\,.
\end{equation}
Substituting here (\ref{ximexpS}) this can be expanded in powers of 
$\sqrt{\tau}$ as
\begin{equation}\label{Imgexp2}
\frac{\sinh \xi}{h'(\xi)} = \frac{g_0}{\sqrt{\tau}} + g_1 +
g_2 \sqrt{\tau}+ O(\tau) \,.
\end{equation}
As in the previous case, $g_0, g_2, \cdots $ are imaginary,
while $g_1, g_3, \cdots$ are real, such that only the even order terms
contribute.

By Watson's lemma \cite{Olver} we can exchange
expansion and integration, and we get the leading asymptotic term 
\begin{equation}
Im \int_0^\infty e^{-\frac{1}{t} \tau} \frac{\sinh \xi}{h'(\xi)} d\tau =
\frac{\sin y_1}{\sqrt{2 h''(S)}} \left( \sqrt{\pi t} + O(t^{3/2}) \right) \,.
\end{equation}
The $O(t^{3/2})$ subleading term is given in explicit form in the Appendix. 

iii) $\rho =1$. 
Some details of proof are different for $\rho=1$ since the saddle point $S$ at $\xi=i\pi$ is
of fourth order, so the Taylor expansion of $\tau$ starts with the fourth order term in $\xi-S$
\begin{equation}
\tau = h(\xi) - h(S) = \frac{1}{4!} h''''(S) (\xi - S)^4 + \cdots\,, \quad h''''(S)=-1
\end{equation}
This expansion in $z:= \xi - S$ can be put in closed form as
\begin{equation}
\tau = - \frac{1}{4!} z^4 - \frac{1}{6!} z^6 + O(z^8) = - \cosh z + 1 + \frac12 z^2\,.
\end{equation}
The inversion of this expansion gives
\begin{eqnarray}\label{Puiseux}
z = \xi - S &=&  (-24\tau)^{1/4} - \frac{1}{120} (-24 \tau)^{3/4}  
+ \frac{11}{67200} (-24 \tau)^{5/4}  + O(\tau^{7/4}) \,.
\end{eqnarray}

Deform the $\xi:(-\infty,+\infty)$ integration contour as in the right plot of Fig.~\ref{Fig:contours} along the steepest descent curve $\mbox{Im}(h(\xi))=0$. 
Denote this curve $y_0(x)$, where $y_0(x)$ is the solution satisfying $0<y<\pi$ 
of $\sinh x \sin y = (\pi - y) x$. As in the previous cases, it is sufficient to evaluate only the $[S,\infty)$ integral, 
which has the same form as (\ref{IScase2}).

The expansion of the integrand starts with an inverse quadratic term
\begin{equation}\label{integrand2}
\frac{\sinh \xi}{h'(\xi)} = \frac{\cosh S (\xi-S) + \cdots }{\frac{1}{3!} h''''(S) (\xi-S)^3 + \cdots} =
  \frac{6}{(\xi - S)^2 }  + O((\xi-S)^3)
\end{equation}
Substituting the derivatives $h^{(k)}(S)$ gives an explicit expression for the expansion in $z:=\xi-S$
\begin{equation}
\frac{\sinh \xi}{h'(\xi)} = \frac{- z - \frac{1}{3!} z^3 - \frac{1}{5!} z^5 - O(z^7)}{- \frac{1}{3!} z^3 - \frac{1}{5!} z^5 - O(z^7)} = \frac{\sinh z}{\sinh z - z} \,.
\end{equation}

Substituting the expansion (\ref{Puiseux}) into (\ref{integrand2}) gives an expansion of the form
\begin{equation}\label{tauexp2}
\frac{\sinh \xi}{h'(\xi)} = 
 \sqrt{\frac32} \frac{1}{\sqrt{-\tau}} + \frac45 + \frac{1}{35} \sqrt{\frac32} \sqrt{-\tau} + O(\tau)
\end{equation}
Along the integration path the square root is defined with a negative imaginary part $\sqrt{-\tau} = - i\sqrt{\tau}$,
which gives
\begin{equation}\label{Imgexp3}
\mbox{Im} \frac{\sinh \xi}{h'(\xi)} = \sqrt{\frac32} \frac{1}{\sqrt{\tau}} - \frac{1}{35} \sqrt{\frac32} \sqrt{\tau} + O(\tau^{3/2})\,.
\end{equation}
Thus the expansion is in powers of $\tau$, rather than fractional powers of $\tau$.

Application of Watson's lemma gives the expansion of the integral
\begin{equation}
\mbox{Im }\int_0^\infty e^{-\frac{1}{t}\tau} \frac{\sinh \xi}{h'(\xi)} d\tau 
=  e^{-\frac{1}{t} h(S)} \sqrt{\frac32 \pi t} \left(1 - \frac{1}{70} t + O(t^2) \right)
\end{equation}
with $h(S) = \frac12 \pi^2-1$. As in the previous cases, the expansion is in integer powers of $t$.
The leading term coincides with the $\rho\to 1$ limit of the above two cases.
The $O(t)$ subleading term reproduces the first term in Eq.~(\ref{tildeg2rho1exp}) in the Appendix.

\end{proof}

The higher order terms $O(t^j)$ with $j>2$ for case (iii) can be easily computed.

\begin{remark}
The expansion for $\rho=1$ takes a particularly simple form.
This is an alternating series, and the first three terms are
\begin{equation}
\theta(1/t,t) = \frac{\sqrt3}{2\pi t} e^{\frac{1}{t}} \left( 1 - \frac{1}{70} t + \frac{749,033}{1,034,880,000} t^2 + O(t^3)\right)
\end{equation}

\end{remark}

\subsection{Error estimates}

We examine here the error of the approximation for $\theta(r,t)$
obtained by keeping only the 
leading asymptotic term in Proposition~\ref{prop:1}, and give an explicit
error bound (\ref{errbound}), backed by the results of a numerical study. 
The bound is uniform in $\rho$. 

Define the remainder of the asymptotic expansion of Proposition~\ref{prop:1} as
\begin{equation}\label{varthetadef}
\theta(\rho/t,t)  = \frac{1}{2\pi t} G(\rho)e^{\frac{1}{t} [F(\rho) - \frac{\pi^2}{2}]}
\left(1 + \vartheta(t,\rho) \right) \,.
\end{equation}

First, we note that for all cases (i) $0<\rho<1$ (ii) $\rho>1$ and (iii) $\rho=1$,
the integrand in the application of Watson's lemma has a similar expansion
\begin{equation}\label{Imgexp}
\mbox{Im} g(\xi) = \frac{\mbox{Im} g_0(\rho)}{\sqrt{\tau}} + \mbox{Im} g_2(\rho) \sqrt{\tau} + O(\tau^{3/2})
\end{equation}
see (\ref{Imgexp1}), (\ref{Imgexp2}) and (\ref{Imgexp3}), respectively.
The coefficient of the leading order term is
\begin{equation}
\mbox{Im} g_0(\rho) = \left\{
\begin{array}{cc}
\frac{\sinh x_1}{\sqrt{2(\rho \cosh x_1 - 1)}} & \,, 0 < \rho < 1 \\
\sqrt{\frac32} & \,, \rho=1 \\
\frac{\sin y_1}{\sqrt{2(1+\rho \cos y_1)}} & \,, \rho > 1 \\
\end{array}
\right.
\end{equation}

We can obtain a bound on the remainder function $\vartheta(t,\rho)$ by 
examining the approximation error of the integrand $\mbox{Im }g(\xi)$ by the first term in its series expansion (\ref{Imgexp}). Denoting $\xi=x+i y$ we have
\begin{equation}
\mbox{Im } g(\xi) = \frac{x \cosh x \sin y + (\pi - y) \cos y \sinh x}
{(\rho \cosh x \sin y - \pi + y)^2 + (x + \rho \cos y \sinh x)^2} \,.
\end{equation}

Define 
\begin{equation}\label{img}
\mbox{Im}\, g(\xi(\tau)) := \frac{\mbox{Im} g_0(\rho)}{\sqrt{\tau}} (1 + \delta(\tau,\rho) )
\end{equation}

Numerical testing shows the following properties of the 
remainder function $\delta(\tau,\rho)$:

i) $\delta(\tau,\rho) \leq 0$ is negative

ii) $\delta(\tau,\rho) \to -1$ as $\tau \to \infty$.
This follows by expanding $Im g(\xi(\tau))$  as $x\to \infty$ along the steepest descent curves $y(x)$. 
For $\rho=1$ this curve has the asymptotics $y(x) \simeq 2\pi x e^{-x}$. 

Expressed in terms of $\tau$ we get
\begin{equation}
Im g(\xi(\tau)) = \frac{\pi}{\tau} + O(\tau^{-2})\,,\quad \tau \to \infty\,.
\end{equation}

iii) $|\delta(\tau,\rho) | \leq 1$.

iv) it is bounded in absolute value as 
\begin{equation}\label{deltabound}
|\delta(\tau,\rho)| \leq |\tilde g_2(\rho)| \tau \leq \frac{1}{35} \tau \,,\quad \tau \geq 0
\end{equation}
for all $\rho > 0$, where $\tilde g_2(\rho) := \frac{Im g_2(\rho)}{Im g_0(\rho)}$. 
An explicit result for $\tilde g_2(\rho)$ is given in Eq.~(\ref{tildeg2}) in the
Appendix, and its plot is shown in Fig.~\ref{Fig:errasympt0}.
Numerical study shows that $|\tilde g_2(\rho)|$ reaches a maximum value at
$\rho=1$, equal to $\frac{1}{35}$. This gives a uniform bound, valid for all $\rho>0$.

We can transfer these results to a bound on the error of the application of Watson's lemma. 

\begin{remark}
The error term $\vartheta(t,\rho)$ in (\ref{varthetadef}) is bounded from above as
\begin{equation}\label{errbound}
|\vartheta(t, \rho)| \leq \frac12\cdot \frac{|Im g_2(\rho)|}{|Im g_0(\rho)|} t 
\leq \frac{1}{70} t\,,
\end{equation}
with $g_0(\rho), g_2(\rho)$ the coefficients in the series expansion (\ref{Imgexp}). The ratio in this bound is denoted as $\frac{|Im g_2(\rho)|}{|Im g_0(\rho)|} = \tilde g_2(\rho)$ where $\tilde g_2(\rho)$ is given in explicit form in Eq.~(\ref{tildeg2}) in the Appendix. It reaches its maximum (in absolute value) at $\rho=1$ with 
$\tilde g_2(1) = -\frac{1}{35}$.

Using the property (iii) $|\delta(\tau,\rho)|\leq 1$, this bound can be improved as
\begin{equation}
|\vartheta(t,\rho)| \leq \min\left(\frac{1}{70} t,1 \right)\,.
\end{equation}

\end{remark}

\begin{proof}

The approximation error introduced by keeping only the leading order term in the application of Watson's lemma is estimated as
\begin{equation}
\int_0^\infty e^{-\frac{1}{t} \tau} Im g(\tau) d\tau =
\int_0^\infty e^{-\frac{1}{t} \tau} \frac{g_0(\rho)}{\sqrt{\tau}} (1 + \delta(\tau, \rho)) 
d\tau:=
Im g_0(\rho) \sqrt{\pi t} ( 1 + \vartheta(t,\rho) )
\end{equation}
In the second equality we substituted (\ref{img}). 
Using the bound (\ref{deltabound}) gives the quoted result.
\end{proof}





\subsection{Properties of the functions $F(\rho)$ and $G(\rho)$}
\label{sec:FG}

Let us examine in more detail the properties of the functions $F(\rho)$
and $G(\rho)$ appearing in Proposition \ref{prop:1}. 
We start by studying the behavior of the solutions of the equations
(\ref{eqx1}) and (\ref{eqy1}) for $x_1,y_1$ respectively. For $0 < \rho \leq 1$
the solution of (\ref{eqx1}) approaches $x_1 \to 0$ as $\rho \to 1$ and
it increases to infinity as $\rho \to 0$. For $\rho \geq 1$, the solution 
of (\ref{eqy1}) starts at $y_1 =\pi$ for $\rho \to 1$, and decreases to
zero as $\rho \to \infty$.

The derivative $F'(\rho)$ can be computed exactly
\begin{eqnarray}\label{Fpsol}
F'(\rho) = \left\{
\begin{array}{cc}
- \cosh x_1 \,, & 0 < \rho < 1 \\
\cos y_1  \,, & \rho > 1 \\
\end{array}
\right.
\end{eqnarray}

This shows that the minimum of $F(\rho)$ is reached for that value of $\rho>1$
for which $y_1=\frac{\pi}{2}$. From (\ref{eqy1}) it follows that this is reached at
$\rho=\frac{\pi}{2}$, and $F(\pi/2)=\frac{3\pi^2}{8}$.

We can obtain also the asymptotics of $F(\rho)$ for very small and very large
arguments.

\begin{proposition}\label{prop:F}

i) As $\rho\to 0$ the function $F(\rho)$ has the asymptotics
\begin{equation}
F(\rho) = \frac12 L^2 + L \log (2L) -
L + \log^2(2L) + \frac{\pi^2}{2}+ o(1)
\end{equation}
with $L=\log(1/\rho)$.

ii) As $\rho\to \infty$ the function $F(\rho)$ has the asymptotics 
\begin{equation}\label{Frhoexp}
F(\rho) =  \rho + \frac{\pi^2}{2(1+\rho)} + O(\rho^{-3})\,.
\end{equation}
 
iii) Around $\rho=1$, the function $F(\rho)$ has the expansion
\begin{eqnarray}\label{F1exp}
F(\rho) &=& \frac{\pi^2}{2} - 1 - (\rho - 1) + \frac32 (\rho - 1)^2 
- \frac65 (\rho-1)^3 + \frac{351}{350} (\rho-1)^4 \\
&& - \frac{108}{125}(\rho - 1)^5 + \frac{372903}{490000}(\rho-1)^6 
+ O((\rho-1)^7)\,.\nonumber
\end{eqnarray}

\end{proposition}

\begin{proof}

i) As $\rho\to 0$, the solution of the equation (\ref{eqx1}) approaches
$x_1\to \infty$ as
\begin{equation}\label{x1asympt}
x_1 = L + \log(2L) +o(1)
\end{equation}
with $L=\log(1/\rho)$. This follows from writing Eq.~(\ref{eqx1}) as
\begin{equation}
\frac{1}{\rho} = \frac{\sinh x_1}{x_1} = \frac{e^{x_1}}{2x_1}(1-e^{-2x_1})
\end{equation}
Taking the logs of both sides gives 
\begin{equation}
x_1 = L + \log(2x_1) - \log(1-e^{-2x_1}) \,.
\end{equation}
This can be iterated starting with the zero-th approximation $x_1^{(0)}=L$,
which gives (\ref{x1asympt}).

Eliminating $\rho$ in terms
of $x_1$ using (\ref{eqx1}) gives
\begin{equation}\label{F1alt}
F(\rho) = \frac12 x_1^2 - \frac{x_1}{\tanh x_1} + \frac{\pi^2}{2}
\end{equation}
Substituting (\ref{x1asympt}) into this expression gives the quoted result.

ii) As $\rho \to \infty$, the solution of the equation (\ref{eqy1})
approaches $y_1\to 0$. This equation is approximated as
\begin{equation}
(1+\rho) y_1 + \frac16 \rho y_1^3 + \rho O(y_1^5) = \pi
\end{equation}
which is inverted as
\begin{equation}\label{y1asympt}
y_1 = \frac{\pi}{1+\rho} - \frac{\pi^3}{6} \frac{\rho}{(1+\rho)^4} + O(\rho^{-5})
\end{equation}
Substituting into $F(\rho)=-\frac12 y_1^2 + \rho \cos y_1 + \pi y_1$ 
gives the result quoted above.

iii) Consider first the case $0<\rho\leq 1$. As $\rho \to 1$, we have
$x_1\to 0$. From (\ref{eqx1}) we get an expansion $x_1^2=-6(\rho-1)
+ \frac{21}{5} (\rho-1)^2 - \frac{564}{175}(\rho-1)^3 + O((\rho-1)^4)$.
Substituting into (\ref{F1alt}) gives the result (\ref{F1exp}). 
The same result result is obtained for the $\rho\geq 1$ case, when $\rho=1$ 
is approached from above.

\end{proof}


We give next the asymptotics of $G(\rho)$. 

\begin{proposition}\label{prop:G}

i) As $\rho\to 0$ the asymptotics of $G(\rho)$ is
\begin{equation}
G(\rho) = \sqrt{L} - \rho^2 \sqrt{L} + \frac{\log 2L + 1}{2\sqrt{L}} + o(1)\,,
\quad L := \log(1/\rho)\,.
\end{equation}

ii) As $\rho\to \infty$ we have
\begin{equation}
G(\rho) = \frac{\pi \rho}{(1+\rho)^{3/2}} + O(\rho^{-5/2})\,.
\end{equation}

iii) Around $\rho=1$ the function $G(\rho)$ has the following expansion
\begin{eqnarray}\label{Gexp}
G(\rho) &=& \sqrt3 \left\{ 1 + \frac15 \left(\frac{1}{\rho} - 1\right) 
- \frac{4}{35} \left(\frac{1}{\rho} - 1\right)^2 
+ \frac{2}{25}
\left(\frac{1}{\rho} - 1\right)^3 + O(\left(\frac{1}{\rho} - 1\right)^4)
\right\}\,.
\end{eqnarray}

\end{proposition}

\begin{proof}

i) The asymptotic expansion of $x_1\to \infty$ for $\rho\to 0$ was already
obtained in (\ref{x1asympt}). Substitute this into 
\begin{equation}
G(\rho) = \frac{x_1}{\sqrt{\frac{x_1}{\tanh x_1}-1}}
\end{equation}
which follows from (\ref{Gsol}) by eliminating $\rho$ using (\ref{eqx1}).
Expanding the result gives the result quoted.

ii) 
Eliminating $\rho$ from the expression (\ref{Gsol}) for $G(\rho)$ for 
$\rho\geq 1$ gives
\begin{equation}
G(\rho) = \frac{\pi - y_1}{\sqrt{1 + \frac{\pi-y_1}{\tan y_1}}}
\end{equation}
Using here the expansion of $y_1$ from (\ref{y1asympt}) gives the result quoted.

iii) The proof is similar to that of point (iii) in Proposition \ref{prop:F}.
\end{proof}


\section{Numerical tests}
\label{sec:3}

The leading term of the expansion in Proposition \ref{prop:1} can be
considered as an approximation for $\theta(r,t)$ for all $t\geq 0$, 
by substituting $\rho=rt$.
Consider the approximation for $\theta(r,t)$ defined as
\begin{equation}\label{tapprox}
\hat \theta(r,t) := \frac{1}{2\pi t} 
G(rt) e^{-\frac{1}{t}(F(rt) - \frac{\pi^2}{2})}\,.
\end{equation}

We would like to compare this approximation with the leading asymptotic 
expansion of \cite{Gerhold2011}, which is obtained by taking the limit $t\to 0$ 
at fixed $r$.
This is given by Theorem 1 in \cite{Gerhold2011}
\begin{equation}\label{G4}
\theta_G(r,t)=\frac{\sqrt{e}}{\pi} \sqrt{\frac{u_0(t)}{\log u_0(t)-2-2\rho}}
e^{-t u_0(t) + \sqrt{2u_0(t)}}
\end{equation}
where $u_0(t)$ is the solution of the equation
\begin{equation}
t = \frac{\log u}{2\sqrt{2u}} - \frac{\rho}{\sqrt{2u}} + \frac{1}{4u}\,,\quad 
\rho := \log\frac{r}{2\sqrt2}\,.
\end{equation}
The correction to Eq.~(\ref{G4}) is of order $ 1 + O(\sqrt{t} \log^2(1/t))$.

Table~\ref{tab:2} shows the numerical evaluation of 
$\hat \theta(r,t)$ for $r=0.5$ and several values of $t$, 
comparing also
with the small $t$ expansion $\theta_G(r,t)$ in Eq.~(\ref{G4}).
They agree well for sufficiently small $t$ but start to diverge for 
larger $t$. In the last column of Table~\ref{tab:2} we show also the result
of a direct numerical evaluation of $\theta(r,t)$ by numerical integration
in \textit{Mathematica}. 

\begin{table}
\caption{\label{tab:2} 
Numerical evaluation of $\theta(0.5,t)$ using the asymptotic expansion
of Proposition \ref{prop:1} $\hat\theta(0.5,t)$ and the Gerhold approximation
$\theta_G(0.5,t)$ given in (\ref{G4}). The last column shows the result
of a direct numerical evaluation using numerical integration in 
\textit{Mathematica}.}
\begin{center}
\begin{tabular}{|c|cccc|cc|c|}
\hline
$t$ & $\rho$ & $x_1$ & $F(\rho)$ & $\hat\theta(0.5,t)$ 
    & $u_0(t)$ & $\theta_G(0.5,t)$ & $\theta_{\rm num}(0.5,t)$ \\
\hline\hline
0.1 & 0.05 & 5.3697 & 13.9816 & $2.098 \cdot 10^{-39}$ 
    & 1447.8 & $2.101\cdot 10^{-39}$ & - \\
0.2 & 0.1  & 4.4999 & 10.5584 & $1.176 \cdot 10^{-12}$ 
    & 256.3 & $1.181\cdot 10^{-12}$ & $1.173\cdot 10^{-12}$ \\
0.3 & 0.15  & 3.9692 & 8.84 & $2.713 \cdot 10^{-6}$ 
    & 89.713 & $2.738\cdot 10^{-6}$ & $2.704\cdot 10^{-6}$ \\
0.5 & 0.25 & 3.2638 & 6.9876  & 0.0114 
    & 22.69 & 0.0116 & 0.0113 \\
1.0 & 0.5  & 2.1773 & 5.0712  & 0.2722 
    & 3.1345 & 0.3062 & 0.2685 \\
1.5 & 0.75 & 1.3512 & 4.3023  & 0.2960  
    & 0.9531 & 0.4097 & 0.2900 \\
2.0 & 1.0  & 0.0 & 3.9348     & 0.2300 
    & 0.4271 & 0.4690 & 0.2213 \\
\hline
$t$ & $\rho$ & $y_1$ & $F(\rho)$ & $\hat\theta(0.5,t)$ 
    & $u_0(t)$ & $\theta_G(0.5,t)$ & $\theta_{\rm num}(0.5,t)$ \\
\hline
2.5 & 1.25  & 2.0105 & 3.7630 & 0.1682 & 0.2430 & 1.2541 & 0.1628 \\
3.0 & 1.5   & 1.6458 & 3.7037 & 0.127  & 0.1607 & - & 0.1222 \\
10.0 & 5.0 & 0.5459 & 5.8393 & 0.0164  & 0.0234 & - & 0.0151 \\
\hline
\end{tabular}
\end{center}
\end{table}

Figure~\ref{Fig:3G} compares the approximation $\hat \theta(r,t)$ (black
curves) against $\theta_G(r,t)$ (blue curves) and direct numerical
integration (red curves). We note the same pattern of good agreement
between $\hat \theta(r,t)$ and $\theta_G(r,t)$ at small $t$, and increasing
discrepancy for larger $t$, which explodes to infinity for sufficiently
large $t$. The reason for this
explosive behavior is the fact that the denominator in Eq.~(\ref{G4}) approaches
zero as $t$ approaches a certain maximum value. For $t$ larger than 
this value the denominator becomes negative and the approximation ceases to
exist.

This maximum $t$ value is given by the inequality
$u_0(t) < e^{-2-2\log(r/(2\sqrt{2}))}=\frac{8}{r^2 e^2}$.
For $r=0.5$ this condition imposes the upper bound $t< t_{\rm max}=2.5538$. 

\begin{figure}
    \centering
   \includegraphics[width=1.7in]{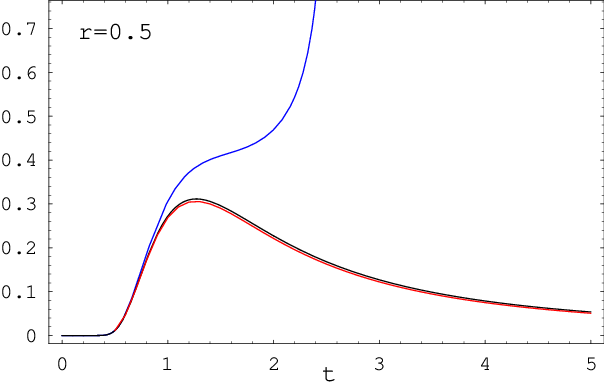}
   \includegraphics[width=1.7in]{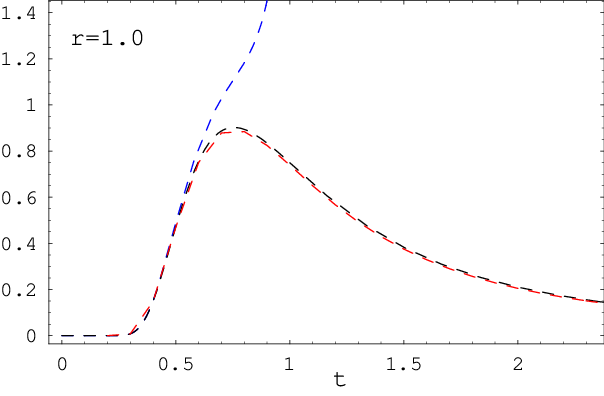}
   \includegraphics[width=1.7in]{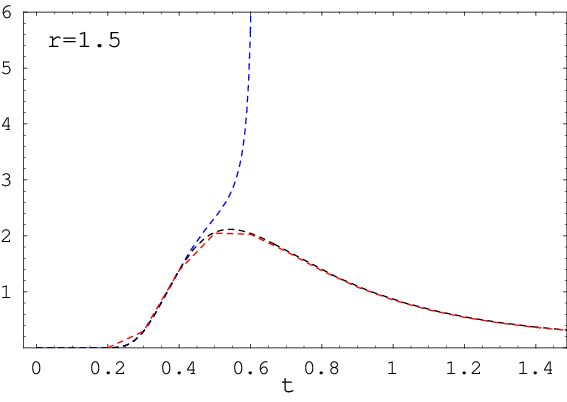}
    \caption{ Plot of $\hat\theta(r,t)$ vs $t$ from the asymptotic 
expansion of Proposition \ref{prop:1} (black) and from the Theorem 1
of Gerhold \cite{Gerhold2011} (blue). The red curves show the results
of direct numerical integration of $\theta(r,t)$.
The three panels correspond to the three values of $r=0.5, 1.0,1.5$. }
\label{Fig:3G}
 \end{figure}

\begin{figure}
    \centering
      \includegraphics[width=4in]{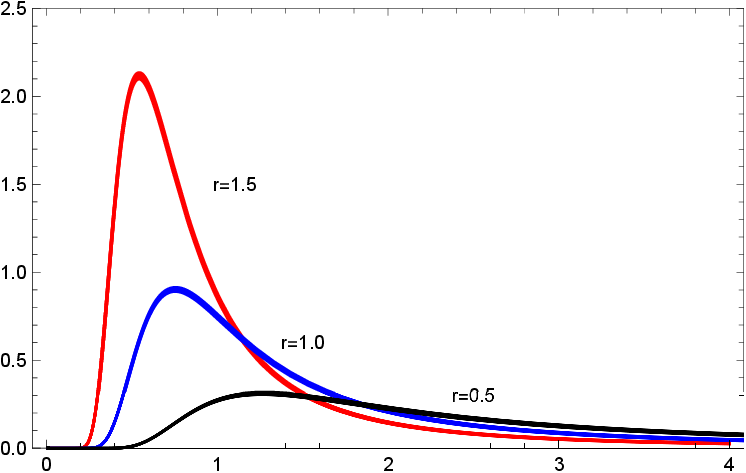}
    \caption{ Plot of $\hat \theta(r,t)$ vs $t$ defined in (\ref{tapprox})
giving the leading asymptotic result 
from Proposition \ref{prop:1} for three values of $r=0.5, 1.0,1.5$
(solid, dashed, dotted). The widths of the curves reflect the error
bound (\ref{errbound}).}
\label{Fig:3t}
 \end{figure}

We show in Figure \ref{Fig:3t} plots of $\hat \theta(r,t)$ vs $t$ at 
$r=0.5,1.0,1.5$. The vertical scale of the plot is chosen as in Figure 1 
(left) of Bernhart and Mai \cite{Bernhart2015}, which shows the plots of 
$\theta(r,t)$ for the same parameters, obtained using a precise
numerical inversion of the Laplace transform of $\theta(r,t)$.
The shapes of the curves are very similar with those shown in \cite{Bernhart2015}.


\subsection{Asymptotics for $t\to 0$ and $t\to \infty$ of the approximation
$\hat \theta(r,t)$}

We study in this Section the asymptotics of the approximation 
$\hat \theta(r,t)$ defined in (\ref{tapprox}) for $t\to 0$ and $t\to \infty$, 
and compare with the exact asymptotics of $\theta(r,t)$ in these limits 
obtained by Barrieu et al \cite{Barrieu2004} and Gerhold \cite{Gerhold2011}. 

As expected, the $t\to 0$ asymptotics matches precisely
the exact asymptotics of $\theta(r,t)$.
Although the approximation $\hat \theta(r,t)$ was derived in the small $t$
limit, it is somewhat surprising that it matches also the correct asymptotics
of $\theta(r,t)$ for $t\to \infty$, up to a multiplicative coefficient,
which however becomes exact as $r\to \infty$.

\textit{Small $t$ asymptotics $t\to 0$.}
The leading asymptotics of $\theta(r,t)$ as $t\to 0$ was obtained 
in Barrieu et al \cite{Barrieu2004} 
\begin{equation}
\theta(r,t) \sim e^{-\frac{1}{2t} \log^2 t} \,.
\end{equation}
This was refined further by Gerhold as in Eq.~(\ref{G4}). 
Using the small $t$ approximation $u_0(t) = \frac{1}{2t^2} \log^2(1/t)$
(see Eq.~(6) in Gerhold's paper), the improved expansion (\ref{G4}) gives
\begin{equation}
\theta_G(r,t) \sim \frac{\log(1/t)}{t} 
\frac{1}{\sqrt{2\log\log(1/t)+2\log(2/t)}}
e^{-\frac{1}{2t} \log^2(1/t)} \,.
\end{equation}

The $t\to 0$ expansion of $\hat \theta(r,t)$ can be obtained using
the $\rho\to 0$ asymptotics of $F(\rho),G(\rho)$ in points (i) of the
Propositions~\ref{prop:F} and \ref{prop:G}, respectively. 
This gives
\begin{equation}
\hat\theta(r,t) \sim \frac{1}{t} \sqrt{\log 1/(rt)} 
e^{-\frac{1}{2t} \log^2(r t) - \frac{1}{t} \log(1/(rt)) \log(2 \log(1/(rt))
+ \frac{1}{t} \log(1/(rt))}\,,\quad t\to 0
\end{equation}
The exponential factor agrees with the asymptotics of the exact result.
Also the leading dependence of the multiplying factor reproduces the 
improved expansion following from \cite{Gerhold2011}.

\textit{Large $t$ asymptotics $t\to \infty$.}
The $t\to \infty$ asymptotics of $\theta(r,t)$ is given in Remark 3 
in Barrieu et al \cite{Barrieu2004} as
\begin{equation}\label{larget}
\theta(r,t) \sim c_r \frac{1}{t^{3/2}}\,,\quad
c_r = \frac{1}{\sqrt{2\pi}} K_0(r) \,.
\end{equation}

The large $t$ asymptotics of $\hat \theta(r,t)$ is related to the 
$\rho\to \infty$ asymptotics of $F(\rho),G(\rho)$.
As $\rho\to \infty$ we have $F(\rho) \sim \rho$ from Prop.~\ref{prop:F} 
point (ii)
and $G(\rho) \sim \frac{\pi \rho}{(1+\rho)^{3/2}}$ from Prop.~\ref{prop:G} 
point (ii).
Substituting into (\ref{tapprox}) and keeping only the leading
contributions as $t\to \infty$ gives
\begin{equation}\label{51}
\hat \theta(r,t) \sim \frac{r}{2 (1+rt)^{3/2}} e^{-r} \sim
\frac{1}{2\sqrt{r}} t^{-3/2} e^{-r} \,.
\end{equation}
The $t$ dependence has the same form as the exact asymptotics (\ref{larget}). 

Let us examine also the coefficient of $t^{-3/2}$ in (\ref{51}) and
compare with the exact result for $c_r$ in (\ref{larget}).
The leading asymptotics of the $K_0(x)$ function for $x\to \infty$
is $K(x) = e^{-x} \sqrt{\frac{\pi}{2x}}(1+O(1/x))$. The exact coefficient
becomes, for $r\to \infty$
\begin{equation}
c_r = \frac12 \frac{1}{\sqrt{r}} e^{-r} + O(1/r)
\end{equation}
The leading term in this expansion matches precisely the coefficient obtained 
from $\hat\theta(r,t)$. 
We conclude that the right tail asymptotics of $\hat \theta(r,t)$ approaches
the same form as the exact asymptotics in the limit $r \to \infty$.

\section{Application: The asymptotic distribution of $\frac{1}{t} A_t^{(\mu)}$}
\label{sec:4}

As an application of the result of Proposition~\ref{prop:1}, we give here
the leading $t\to 0$ asymptotics of the distribution of the time average of the
gBM $\frac{1}{t} A_t^{(\mu)}$.

The distribution of $\frac{1}{t} A_t^{(\mu)}$ can be obtained from Yor's formula 
(\ref{AB}) by integration over $x$.
Introducing a new variable $u = a t$ this is expressed as the integral
\begin{equation}\label{A}
\mathbb{P}\left(\frac{1}{t} A_t^{(\mu)} \in da \right) = 
\int_{x\in \mathbb{R}} e^{\mu x - \frac12 \mu^2 t}
\exp\left( - \frac{1+e^{2x}}{2a t} \right) 
\theta\left(\frac{e^x}{a t},t\right) \frac{da dx}{a}\,.
\end{equation}
Changing the $x$ integration variable to $x= \log(a\rho)$,
the range of integration for $x:(-\infty,\infty)$
is mapped to $\rho:(0,\infty)$. 

\begin{eqnarray}\label{Agen}
\mathbb{P}\left(\frac{1}{t} A_t^{(\mu)} \in da\right) = 
e^{-\frac12\mu^2 t} \frac{da}{a} \int_0^\infty (a \rho)^\mu 
e^{-\frac{1+a^2 \rho^2}{2a t}}
\theta\left(\frac{\rho}{t},t\right)
\frac{d\rho}{\rho} \,.
\end{eqnarray}

We would like to use here the asymptotic expansion for $\theta(\rho/t,t)$ 
from Proposition~\ref{prop:1}. 
Integrating term by term the asymptotic expansion requires a uniform bound on the approximation error. 
Such a bound was given in (\ref{errbound}) based on numerical evidence,
which is furthermore uniform in $\rho$.
Using this result one can transfer the asymptotic expansion
of Proposition~\ref{prop:1} to an asymptotic result for the density of the time-average
of the gBM. Define this density as
\begin{equation}
\mathbb{P}\left(\frac{1}{t} A_t^{(\mu)} \in da\right) = f(a,t) \frac{da}{a} \,.\nonumber
\end{equation}


\begin{proposition}\label{prop:smallA}
Assume the error bound (\ref{errbound}). 
Then we have 
\begin{equation}\label{Apdf}
f(a,t) = \frac{1}{2\pi t} e^{-\frac12 \mu^2 t} \int_0^\infty (a\rho)^\mu e^{-\frac{1}{t} H(\rho)} G(\rho) \frac{d\rho}{\rho} \cdot (1 + \varepsilon(a,t) )
\end{equation}
with 
\begin{equation}
H(\rho) := \frac{1+a^2 \rho^2}{2a} -\frac{\pi^2}{2} + F(\rho) \,.
\end{equation}
The error term is bounded as $|\varepsilon(a,t) |\leq \frac{1}{70} t$.

Furthermore, the density $f(a,t)$ has the asymptotic form in the small-maturity limit
\begin{equation}\label{fLaplace}
f(a,t) \sim \frac{1}{\sqrt{2\pi t}}
g(a,\mu)  e^{- \frac{1}{t} J(a)} \,,\quad t\to 0\,,
\end{equation}
with $J(a) \equiv \inf_{\rho\geq 0} H(\rho) = H(\rho_*)$ and
\begin{equation}\label{gadef}
g(a,\mu) := 
(a\rho_*)^\mu G(\rho_*) \frac{1}{\rho_* \sqrt{H''(\rho_*)}} \,,\quad
\rho_* = \arg\inf_\rho H(\rho)\,.
\end{equation}

\end{proposition}

\begin{proof}
The function $\theta(\rho,t)$ can be expressed as the 
leading order asymptotic term in Proposition~\ref{prop:1} 
\begin{equation}
\theta(\rho/t,t) = \frac{1}{2\pi t} e^{-\frac{1}{t}(F(\rho) - \frac{\pi^2}{2})} G(\rho) 
\Big(1 + \vartheta(\rho,t)\Big) \,,
\end{equation}
where the error term
$\vartheta(\rho,t)$ is bounded as in (\ref{errbound}).

Substituting into the integral in (\ref{Agen}) gives
\begin{eqnarray}\label{Agen2}
&& \int_0^\infty (a \rho)^\mu 
e^{-\frac{1+a^2 \rho^2}{2a t}}
\theta\left(\frac{\rho}{t},t\right)
\frac{d\rho}{\rho} =
\frac{1}{2\pi t} \int_0^\infty (a \rho)^\mu 
e^{-\frac{1}{t} H(\rho)} G(\rho)
\frac{d\rho}{\rho} + I_\varepsilon(a,t) 
\end{eqnarray}
with $\rho_* =\arg\inf_\rho H(\rho)$.

The term $I_\varepsilon(a,t)$ is bounded using the error bound (\ref{errbound})
\begin{eqnarray}\label{Ainterr2}
&& |I_\varepsilon(a,t) | \leq \int_0^\infty 
(a \rho)^\mu 
e^{-\frac{1+a^2\rho^2}{2 a t}} \hat \theta(\rho,t) |\vartheta(\rho,t) |
\frac{d\rho}{\rho}  \leq \frac{1}{70} t \int_0^\infty 
(a \rho)^\mu 
e^{-\frac{1+a^2\rho^2}{2 a t}} \hat \theta(\rho,t) 
\frac{d\rho}{\rho} 
\end{eqnarray}
which gives $|\varepsilon(t) | \leq \frac{1}{70} t$, as stated.
 
The asymptotic result (\ref{fLaplace}) follows by application of the Laplace asymptotic expansion \cite{Erdelyi} to the integral (\ref{Apdf}).
\end{proof}


\subsection{Properties of $J(a)$}
\label{sec:4J}

In this section we give a more explicit result for the exponential factor
$J(a)$ in the leading asymptotic density (\ref{Apdf}), and study its properties.

\begin{proposition}\label{prop:J}
The function $J(a)$ is given by:

i) for $a\geq 1$ we have
\begin{eqnarray}
J(a) = \frac{1}{2a} - \frac{1}{2a} \cosh^2 x_1 + \frac12 x_1^2
=\frac12 x_1^2 - \frac12 x_1 \tanh x_1
\end{eqnarray}
where $x_1$ is the solution of the equation
\begin{equation}
\frac{\sinh 2x_1}{2x_1} =a \,.
\end{equation}
This case corresponds to $0< \rho_* \leq 1$.

ii) for $0 < a\leq 1$ we have
\begin{equation}
J(a) = \frac12 (y_1-\pi) \tan(y_1-\pi) - \frac12 (y_1-\pi)^2
\end{equation}
where $y_1$ is the solution of the equation
\begin{equation}
y_1 - \frac{1}{a} \sin y_1 \cos y_1=\pi \,.
\end{equation}
This case corresponds to $1 \leq \rho_* < \frac{\pi}{2}$.

\end{proposition}

\begin{proof}

We start with the expression
\begin{eqnarray}\label{Jinf}
J(a) &=& \inf_{\rho \geq 0} \left( \frac{1+a^2 \rho^2}{2a} -\frac{\pi^2}{2}
+ F(\rho) \right) \\
&=& \frac{1}{2a} + \frac{a}{2} \rho_*^2 - \frac{\pi^2}{2}
+ F(\rho_*) \nonumber
\end{eqnarray}
where $\rho_*$ is the solution of the equation
\begin{equation}\label{eqrho}
F'(\rho) + a \rho  = 0\,.
\end{equation}

Using the explicit result for $F'(\rho)$ in Eq.~(\ref{Fpsol}) we get
$F'(\rho)< 0$ for $0 < \rho < \frac{\pi}{2}$ and
$F'(\rho) > 0$ for $\rho > \frac{\pi}{2}$.
This implies that the
solution of the equation (\ref{eqrho}) must satisfy $\rho_*\leq \frac{\pi}{2}$.
As $a\to 0$ the minimizer approaches the upper boundary 
$\lim_{a\to 0} \rho_*=\frac{\pi}{2}$ as $F'(\frac{\pi}{2})=0$.

The solution of (\ref{eqrho}) for $a=1$ is $\rho_*=1$, since $F'(1)=-1$,
from Prop.~\ref{prop:F} point (iii).

The strategy of the proof is to eliminate $\rho_*$ using (\ref{Fpsol}) in terms
of $x_1$ and $y_1$, respectively.
Substituting (\ref{Fpsol}) into (\ref{eqrho}) gives 
\begin{eqnarray}\label{31}
0 < \rho < 1 &:& \cosh x_1 = a \rho \\
\rho > 1 &:& \cos y_1 = - a \rho \,.
\end{eqnarray}

These equations can be used to eliminate $\rho_*$ from the expression of 
$F(\rho_*)$. Substituting into (\ref{Jinf}) reproduces the results shown.

The function $\rho_*(a)$ has the expansion around $a=1$
\begin{equation}\label{rhostar}
\rho_*(a) = 1 - \frac14(a-1) + \frac{19}{160} (a-1)^2 - \frac{1511}{22400}
(a-1)^3 + O((a-1)^4) \,.
\end{equation}
To prove this expansion, assume first $a \to 1$ from above, 
and thus $\rho_* <1$. From Eq.~(\ref{31}) it follows that
\begin{equation}
\rho_* = \frac{1}{a} \cosh x_1(a)
\end{equation}
where $x_1(a)$ is the solution of $\frac{\sinh 2x_1}{2x_1}=a$. 
Expanding around $a=1$ gives $x_1(a) = \sqrt{\frac32(a-1)}  + \cdots$, 
and substituting
above gives the expansion of $\rho_*(a)$. A similar result is obtained by
assuming $a\to 1$ from below.

\end{proof}

The function $J(a)$ has the following properties:

i) The function $J(a)$ vanishes at $a=1$. The infimum in (\ref{Jinf})
is reached at $\rho_*=1$. 

ii) For $a> 1$ the infimum is reached at $0 < \rho_* < 1$. 
As $a \to \infty$ we have $\lim_{a\to \infty} \rho_* = 0$.

iii) For $0 < a < 1$ the infimum is reached at $1 < \rho_* < \rho_0=\frac{\pi}{2}$. 
As $a \to 0$ we have $\lim_{a\to 0} \rho_* = \pi/2$. 

\subsection{Comparison with the Large Deviations result}

The distributional properties of $A_t^{(\mu)}$ as $t\to 0$ were studied using 
probabilistic methods from Large Deviations theory in \cite{MLP,SIFIN,PZAsianCEV}. 
Assuming that $S_t$
follows a one-dimensional diffusion $dS_t= \sigma(S_t) S_t dW_t + 
(r-q) S_t dt$, it has been shown
in \cite{SIFIN} that, under weak assumptions on $\sigma(\cdot)$,
$\mathbb{P}\left(\frac{A_t}{t}\in \cdot\right)$ 
satisfies a Large Deviations property as $t\to 0$.
This result includes as a limiting case the geometric Brownian motion 
corresponding to $\sigma(x)=\sigma$. 

The main result can be extracted from the proof of Theorem 2 in \cite{SIFIN},
and can be formulated as follows. 
\begin{proposition}\label{prop:SIFIN}
Assume $S_t = e^{\sigma W_t + (r-\frac12\sigma^2)t}$ and 
$A_t = \int_0^t S_t dt$. The time average of $S_t$ satisfies a Large
Deviations property in the $t\to 0$ limit
\begin{equation}\label{LDP}
\lim_{t\to 0}t \log \mathbb{P}\left(\frac{A_t}{t} \in \cdot \right)=
- \frac{1}{\sigma^2} \mathcal{J}_{\rm BS}(\cdot ) 
\end{equation}
with rate function $\mathcal{J}_{\rm BS}(\cdot ) $ expressed
as the solution of a variational problem, which is furthermore 
solved in closed form in Proposition 12 of \cite{SIFIN}.
\end{proposition}

The asymptotic density (\ref{Apdf}) has the same form as the exponential decay 
predicted by Large Deviations theory (\ref{LDP}). 
We show here that the result for $J(a)$ in Proposition \ref{prop:J} is related 
to the rate function $\mathcal{J}_{\rm BS}(a)$ derived in \cite{SIFIN} as 
\begin{equation}\label{Jrel}
J(a) = \frac14 \mathcal{J}_{\rm BS}(a) \,.
\end{equation}
The factor of 1/4 is due to the factor of 2 in the exponent of the
definition of $A_t$ in the Yor formula (\ref{AB}),
which acts as a volatility $\sigma=2$.

The function $\mathcal{J}_{\rm BS}(x)$ is given by Proposition 12 in \cite{SIFIN}
which we reproduce here for convenience.

\begin{proposition}[Prop.~12 in \cite{SIFIN}]
The rate function $\mathcal{J}_{\rm BS}(x)$ is given by
\begin{eqnarray}
\mathcal{J}_{\rm BS}(x)= \left\{
\begin{array}{cc}
\frac12\beta^2 - \beta \tanh \left( \frac{\beta}{2}\right) \,, & x\geq 1 \\
2\xi( \tan \xi - \xi) \,, & x \leq 1 \\
\end{array}
\right.
\end{eqnarray}
where $\beta$ is the solution of the equation
\begin{equation}
\frac{\sinh \beta}{\beta} = x 
\end{equation}
and $\xi$ is the solution in $[0, \pi/2]$ of the equation
\begin{equation}
\frac{\sin 2\xi}{2\xi} = x \,.
\end{equation} 
\end{proposition}

The relation (\ref{Jrel}) follows by identifying $x_1 = \frac12 \beta$
in the first case of Proposition \ref{prop:J}, and $y_1-\pi = \xi$ in
the second case. 

The rate function $\mathcal{J}_{\rm BS}(a)$ for $a \simeq 1$ can be expanded in
a Taylor series in $\log a$
(see Eq.~(36) of \cite{SIFIN})
\begin{equation}\label{JTaylor}
\mathcal{J}_{\rm BS}(a) = \frac32 \log^2 a - \frac{3}{10} \log^3 a 
+ \frac{109}{1400}
\log^4 a + O(\log^5 a) \,.
\end{equation}
This expansion is convenient for numerical evaluation of the rate function.
Proposition 13 in \cite{SIFIN} gives also asymptotic expansions of
$\mathcal{J}_{\rm BS}(a)$ for $a\to 0$ and $a\to \infty$, which can be translated
directly into corresponding asymptotic expansions for $J(a)$.


\subsection{Properties of $g(a,\mu)$}
\label{sec:4g}

We study here in more detail the properties of the function $g(a,\mu)$ 
defined in (\ref{gadef}).
This can be put into a more explicit form as follows.

\begin{proposition}\label{prop:8}
The function $g(a,\mu)$ has the explicit form
\begin{eqnarray}\label{gamuexp}
g(a,\mu) &=& 
e^{\mu \log a + (\mu-1)\log\rho_*(a)} G(\rho_*) \frac{1}{\sqrt{H''(\rho_*)}}
= \frac{\sqrt3}{2} e^{c_1 \log a + 
c_2 \log^2 a + O(\log^3 a)} \,.
\end{eqnarray}
with
\begin{eqnarray}\label{c1}
&& c_1 = \frac34 (\mu+1) - \frac45 \\
\label{c2}
&& c_2 = -\frac{3}{80}(\mu+1) + \frac{57}{1400}\,.
\end{eqnarray}

\end{proposition}

\begin{proof}
The proof follows by combining the expansions around $a=1$ of the factors in 
this expression:

i) $\rho_*(a)$ is given by (\ref{rhostar}) which is written in equivalent
form as
\begin{equation}
\log \rho_*(a) = - \frac14 \log a - \frac{3}{80} \log^2 a + \frac{1}{350}
\log^3 a + O(\log^4 a) \,.
\end{equation}

ii) $G(\rho_*)$ is obtained by substituting the expansion of $\rho_*(a)$
into the expansion of $G(\rho)$ in powers of $\rho-1$ from Proposition~\ref{prop:G}
\begin{eqnarray}
G(\rho_*(a)) &=& \sqrt3 \left(
1 + \frac{1}{20} \log a + \frac{37}{5600} \log^2 a -
\frac{41}{48000} \log^3 a + O(\log^4 a) \right) \\
&=& \sqrt3 \exp\left( \frac{1}{20} \log a + \frac{3}{560} \log^2 a
- \frac{1}{875} \log^3 a + O(\log^4 a) \right) \nonumber
\end{eqnarray}

iii) $H''(\rho_*(a)) = F''(\rho_*(a))+a$ can be evaluated using the
expansion of $F(\rho)$ from Prop.~\ref{prop:F} (iii) which gives
\begin{equation}\label{Hsol}
F''(\rho_*) = 
3 + \frac{9}{5} (a-1) - \frac{18}{175} (a-1)^2 + O((a-1)^3)\,.
\end{equation}
We get
\begin{equation}
H''(\rho_*(a)) = 4 + \frac{14}{5}(a-1) - \frac{18}{175}(a-1)^2 + O((a-1)^3)\,.
\end{equation}
It is convenient to express this in terms of $\log a$ as
\begin{equation}
H''(\rho_*) = 4 \exp\left( \frac{7}{10} \log a + \frac{111}{1400} \log^2 a
+ O(\log^3 a) \right)\,.
\end{equation}
\end{proof}

\begin{remark}\label{rmk:1}
We have $g(1,\mu) = \frac{\sqrt3}{2}$ for all $\mu \in \mathbb{R}$. 
This follows by noting that at $a=1$
we have $\rho_*(1)=1$, and thus the only factors different from 1 are
$G(1) = \sqrt3$ and $H''(1)=4$.
\end{remark}

Numerical evaluation of the function $g(a,\mu)$ shows that for
the driftless gBM case $\mu=-1$ this function is decreasing in $a$, 
and becomes increasing for sufficiently large and positive $\mu$.


\begin{figure}
    \centering
   \includegraphics[width=2.4in]{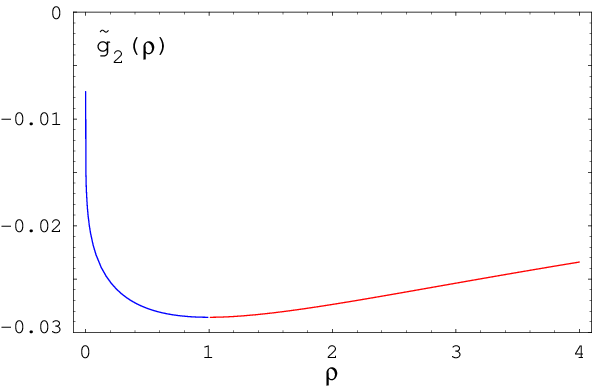}
    \caption{
Plot of the function $\tilde g_2(\rho)$ appearing in the subleading correction to the
asymptotic expansion of the Hartman-Watson integral (\ref{thetasub}).
}
\label{Fig:errasympt0}
 \end{figure}

\vspace{0.2cm}
\textbf{Acknowledgements}\newline

I am grateful to Peter N\'andori and Lingjiong Zhu for very useful discussions and comments.

\appendix

\section{Subleading correction}

We give here the subleading correction to the asymptotic expansion of
the Hartman-Watson integral $\theta(\rho/t,t)$ in Proposition \ref{prop:1}.

The first two terms of the asymptotic expansion of the Hartman-Watson
integral as $t\to 0$ are
\begin{equation}\label{thetasub}
\theta(\rho/t,t) = 
\frac{1}{2\pi t} G(\rho) e^{-\frac{1}{t}(F(\rho) - \frac{\pi^2}{2})}
\left(1 + \frac12 t \tilde g_2(\rho) + O(t^2) \right)
\end{equation}
where
\begin{eqnarray}\label{tildeg2}
\tilde g_2(\rho) = \left\{
\begin{array}{cc}
\frac{-12 + 9 \rho \cosh x_1 - 2 \rho^2 \cosh^2 x_1+5\rho^2}
{12(\rho \cosh x_1 - 1)^3} & \,, 0 < \rho \leq 1 \\
\frac{12 + 9 \rho \cos y_1 + 2 \rho^2 \cos^2 y_1-5\rho^2}
{12(1+\rho \cos y_1)^3} & \,, \rho > 1 \\
\end{array}
\right.
\end{eqnarray}
The result follows straighforwardly by keeping the next terms in the
expansions of the integrals appearing in the proof of Proposition~\ref{prop:1}.

The plot of $\tilde g_2(\rho)$ is given in the left panel of 
Figure~\ref{Fig:errasympt0}. It reaches a maximum (in absolute value) at $\rho=1$
where it is equal to $\tilde g_2(1)=-\frac{1}{35}$.

We list next a few properties of the function $\tilde g_2(\rho)$.
This function has the expansion around $\rho=1$ 
\begin{equation}\label{tildeg2rho1exp}
\tilde g_2(\rho) = 
- \frac{1}{35} + \frac{144}{67375} (1/\rho-1) + O((1/\rho-1)^2) \,.
\end{equation}
As $\rho\to \infty$ we get, using the asymptotics of $y_1\to 0$ from
the proof of Prop.~\ref{prop:F}(ii),
\begin{equation}
\tilde g_2(\rho) = - \frac{1}{4\rho} + \frac{3}{2\rho^2} + O(\rho^{-3})\,.
\end{equation}
For $\rho\to 0$, recall from the proof of Proposition \ref{prop:F} (i) 
that $\rho \cosh x_1 \sim \log(1/\rho)\to \infty$, which gives the asymptotics
\begin{equation}
\tilde g_2(\rho) = -\frac16 \frac{1}{\log(1/\rho)} + O(\log^{-2}(1/\rho))\,.
\end{equation}



\end{document}